\documentclass[11pt]{amsart}
\usepackage{latexsym, fancyhdr, stmaryrd}

\usepackage{amsfonts, amssymb, amsthm, amsmath, hyperref, graphicx, bbm}
\usepackage[mathscr]{eucal}

\newcommand{\ring}[1]{\ensuremath{\mathbb{#1}}}
\newcommand\ZZ{\ring{Z}}
\renewcommand\>{\rangle}
\newcommand\<{\langle}

\newtheorem{theorem}{Theorem}[section]

\newtheorem{corollary}[theorem]{Corollary}

\newtheorem{proposition}[theorem]{Proposition}

\newtheorem{lemma}[theorem]{Lemma}

\theoremstyle{definition}
\newtheorem{example}[theorem]{Example}

\newtheorem{definition}[theorem]{Definition}
\newtheorem{definitions}[theorem]{Definitions}

\newcommand{\dis}{\mathbf{d}}

\newcommand{\ld}{\mathrm{LD}}
\newcommand{\block}{\mathcal{B}}
\DeclareMathOperator\Betti{Betti} 

\begin{document}
\title{}

\title[On Length Densities]{On Length Densities}

\author{Scott T. Chapman}
\address{Department of Mathematics and Statistics, Sam Houston State University, 
Huntsville, TX  77341}
\email{scott.chapman@shsu.edu}
\urladdr{www.shsu.edu/$\sim$stc008/} 

\author{Christopher O'Neill}
\address{Mathematics and Statistics Department, San Diego State University, 
 San Diego, CA 92182}
\email{cdoneill@sdsu.edu}

\author{Vadim Ponomarenko}
\address{Mathematics and Statistics Department, San Diego State University, 
 San Diego, CA 92182}
\email{vadim123@gmail.com}

\begin{abstract}
For a commutative cancellative monoid $M$, we introduce the notion of the \textit{length density} of both a nonunit $x\in M$, denoted $\ld (x)$, and the entire monoid $M$, denoted $\ld (M)$.  This invariant is related to three widely studied invariants in the theory of non-unit factorizations, $L(x)$, $\ell(x)$, and $\rho(x)$.  We consider some general properties of $\ld (x)$ and $\ld (M)$ and give a wide variety of examples using numerical semigroups, Puiseux monoids, and Krull monoids.  While we give an example of a monoid $M$ with irrational length density, we show that if $M$ is finitely generated, then $\ld (M)$ is rational and there is a nonunit element $x\in M$ with $\ld (M)=\ld (x)$ (such a monoid is said to have accepted length density).  While it is well-known that the much studied asymptotic versions of $L(x)$, $\ell (x)$ and $\rho (x)$ (denoted $\overline{L}(x)$, $\overline{\ell}(x)$, and $\overline{\rho} (x)$) always exist, we show the somewhat surprising result that $\overline{\ld}(x) = \lim_{n\rightarrow \infty} \ld (x^n)$ may not exist.  We also give some finiteness conditions on $M$ that force the existence of $\overline{\ld}(x)$.
\end{abstract}



\maketitle

\medskip

\section{Introduction}
\label{sec:intro}

\markright{On Length Densities}

A commutative cancellative monoid $M$ with set of irreducible elements (or atoms) $\mathcal{A}(M)$ is called \textit{atomic} if for each nonunit $x\in M$ there are 
$x_1,\ldots, x_k\in\mathcal{A}(M)$ such that $x=x_1\cdots x_k$.  For such an $x$, set
\begin{equation}\label{lengthset}
\mathsf{L}(x)=\{k\in \mathbb{N}\; |\; \mbox{there exist atoms }x_1,\ldots, x_k \mbox{ such that }x=x_1\cdots x_k\}.
\end{equation}
The set $\mathsf{L}(x)$ is known as the \textit{set of lengths} of $x\in M$ (see \cite{Ger1}), and its study over the past 60 years has been the principal focus of non-unique factorization theory (the monograph \cite{GHKb} is a good general source on this topic).  Several of the fundamental results in this area have appeared in these \textsc{Proceedings}.   
For instance, if $M$ is the multiplicative monoid of an integral domain $R$ then set
\begin{equation}\label{contstants}
L(x)=\max\mathsf{L}(x), 
\quad
\ell(x)=\min\mathsf{L}(x), 
\quad
\rho(x) = \frac{L(x)}{\ell(x)},
\quad \mbox{and} \quad
\rho(M) = \sup\{\rho(x)\; |\; x\in M\}.
\hspace{-0.15in}
\end{equation}
The constant $\rho(x)$ is known as the \textit{elasticity} of $x$ in $M$ and the constant $\rho(M)$ as the \textit{elasticity of} $M$.
If $R$ is an algebraic number ring, then Carlitz  showed in \cite{Ca} that $R$ has class number less than or equal to two if and only if $\rho(M)=1$.
If we further set
\begin{equation}\label{bigLl}
\overline{L}(x) = \lim_{n\rightarrow \infty}\frac{L(x^n)}{n}
\quad \mbox{and} \quad
\overline{\ell}(x) = \lim_{n\rightarrow \infty}\frac{\ell(x^n)}{n},
\end{equation}
then Anderson and Pruis show in \cite{DA} that 
\begin{enumerate}
\item[(i)] both the limits $\overline{L}(x)$ and $\overline{\ell}(x)$ exist (although $\overline{L}(x)$
may be infinite);
\item[(ii)] if $\alpha$ and $\beta\in[0,\infty]$ with $0\leq \alpha\leq 1\leq \beta\leq\infty$, then there is an integral domain $R$ and an irreducible $x\in R$
with $\overline{\ell}(x)=\alpha$ and $\overline{L}(x)=\beta$.
\end{enumerate}
The elasticity is further pursued in \cite{AACS} where it is shown that if $M$ is the multiplicative monoid of a Krull domain with finite divisor class group,  then $\rho(M)$ is rational and moreover there exists a nonunit $x\in M$ so that $\rho(M)=\rho(x)$ (such a monoid is said to have \textit{accepted elasticity}).  

These are but a few of the numerous constants that have been attached to $M$ and $x$ to better describe their factorization properties.  The above constants are rather ``coarse'' in the sense that they merely describe the extreme values in $\mathsf{L}(x)$.  The purpose of this note is to introduce a new constant, finer in nature, which describes not just extremes, but the entire length set. 
\begin{definitions}\label{basic}
Let $M$ be a commutative cancellative atomic BF-monoid with set of units $M^\times$. 
Define a function $\mathsf{L}^\Delta : M \rightarrow \mathbb{N}_0$ via        
\[
\mathsf{L}^\Delta(x) = L(x) - \ell(x)
\]
where we define  $\mathsf{L}^\Delta (x) = 0$ if $x\in M^\times$.
We define the \textit{length ideal} of $M$, denoted $M^{LI}$ as the set of elements with nonzero image under $\mathsf{L}^\Delta$.  
For $x\in M^{LI}$ set
\[
\ld (x)= \frac{|\mathsf{L}(x)|-1}{\mathsf{L}^\Delta(x)},
\]
which we call the \textit{length density} of $x$.
Moreover, set 
\[
\ld(M) = \inf\{\ld (x)\,|\, x\in M^{LI}\},
\]
which we call the \textit{length density} of $M$.  If there is an $x\in M^{LI}$ such that 
$\ld(M) = \ld (x)$, then we say that the length density of $M$ is \textit{accepted}.  Set
\[
\overline{\ld}(x) = \lim_{n\rightarrow \infty}\ld (x^n)
\]
to be the \textit{asymptotic length density of} $x$, provided this limit exists.
\end{definitions} 

Notice that $M$ is half-factorial if and only if $M^{LI}$ is empty; we henceforth exclude such monoids from consideration.  Clearly $M^{LI}$ is an ideal, as $M^{LI}M\subseteq M^{LI}$.
Under the hypothesis of Definitions \ref{basic}, each $\ld(x)$ is a rational number in the interval $(0,1]$ and so $0\leq \ld(M)\leq 1$.
Before considering further the inequality $0\leq \ld(M)\leq 1$, we require some additional notation.  We will call a subset $S\subseteq \mathbb{N}$ an \textit{interval} if $S = [\min S, \max S] \cap \mathbb{N}$.  If $\mathsf{L}(x)=\{n_1, n_2, \ldots ,n_k\}$,
where $n_1<n_2<\cdots <n_k$, then the \emph{delta set} of $x$ and $M$ are defined as 
\[
\Delta(x)=\{n_{i+1}-n_i\; |\; 1\leq i<k\}
\quad \mbox{and} \quad
\Delta(M)=\bigcup_{x\in M} \Delta(x),
\]
respectively.  
There is a wealth of literature concerning the delta set of various types of commutative cancellative monoids:\ Krull monoids~\cite{BCRSY,CGP,GS2,GY1,GZ2}; numerical monoids~\cite{BCKR,GGMV,GSLM,GS1}; Puiseux monoids~\cite{CGG,FG}; and arithmetic congruence monoids~\cite{BCS,BCCM}.  We assume the reader has a working knowledge of the terminology and basic properties of these types of monoids.
 
Since our investigations of $\ld (x)$ and $\ld (M)$ will reduce to the study of $\mathsf{L}(x)$, we at times will rely on some established structure theorems for this set.  Let $L\subset \mathbb{Z}$ be finite, $d\in\mathbb{N}$, and $l,M\in\mathbb{N}_0$.  We call $L$ an \textit{almost arithmetic progression} (or \textit{AAP}) with difference $d$, length $l$, and bound $M$ if
\[L=y+(L^\prime \cup L^\ast \cup L^{\prime\prime})\subseteq y+d\mathbb{Z}\]
where $L^\ast=d\mathbb{Z} \cap [0,ld]$, $L^\prime \subseteq [-M,-1]$, $L^{\prime \prime}\subseteq ld+[1,M]$, and $y\in \mathbb{Z}$.  An analysis of monoids and elements whose lengths are almost arithmetic progressions can be found in \cite[Chapter~4]{GHKb}.

We break our remaining work into 3 sections.  In Section~\ref{sec:basicideas}, we review some basic properties of length density and in Proposition~\ref{basicbounds} offer bounds on $\ld (x)$ and $\ld (M)$.  We consider when the extreme values in these bounds are met and offer a wide array of examples of such behavior.  In Section~\ref{sec:accepted}, Proposition~\ref{niceconstruct} allows one to construct monoids with both arbitrary elasticity and length density.  This construction can be used to construct monoids with irrational length density.  We follow this in Theorem~\ref{t:fingenaccepted} by arguing that a finitely generated monoid has rational accepted length density.  We close Section~\ref{sec:accepted} with a brief discussion of the computation of length density for block monoids, and in turn for general Krull monoids.  Section~\ref{sec:asymptotic} begins with Example~\ref{noasym} which illustrates that the asymptotic length density of an element may not exist; this is in stark constrast to the previously mentioned results for $\overline{\ell}(x)$ and $\overline{L}(x)$ (as well as the asymptotic elasticity defined by $\overline{\rho}(x) = \overline{L}(x)/\overline{\ell}(x)$).  We then give conditions in Theorem \ref{asymptotic} on an atomic monoid which guarantee the existence of $\overline{\ld}(x)$, and note that finitely generated monoids and Krull monoids with finite divisor class group satisfy these conditions.  Our work is generally self contained; we direct the reader to~\cite{GHKb} for any undefined terminology or background.

 \section{Basic Ideas and Bounds on the Length Density}
 \label{sec:basicideas}
 
 We open by considering the largest value that $\ld(x)$ can attain.

\begin{proposition}\label{delta1}  Let $M$ be a commutative cancellative atomic monoid and $x\in M^{LI}$.  The following statements are equivalent.
\begin{enumerate}
\item $\ld(x)=1$. 
\item $\mathsf{L}(x)$ is an interval. 
\item $\Delta(x)=\{1\}$.
\end{enumerate}
If all the elements of $M^{LI}$ satisfy any of these conditions, then $M$ necessarily has accepted length density.  
\end{proposition}

\begin{example}
If $\ld(M)=1$, then all elements of its length ideal satisfy the conditions of Proposition \ref{delta1} and $\Delta(M)=\{1\}$.  In the general scheme of factorization theory, such monoids   have appeared in the literature, but have not been widely studied.  Hence, we offer several examples.
\begin{enumerate}
\item 
A numerical monoid is any cofinite additive submonoid of $\mathbb{N}_0$.  Let $M$ be a numerical monoid generated by an interval of integers (i.e., $M=\langle n, n+1, n+2, \ldots, n+k\rangle$ where $k\leq n-1$).  By \cite[Theorem 3.9]{BCKR}, $\Delta (M)=\{1\}$.  This relationship does not work conversely.  Additionally, let $r_1<r_2< ...< r_k$ be natural numbers with $\gcd(r_1,\ldots ,r_k)=1$ and set $n>r_k^2$.  Then, by \cite[Corollary 5.7]{CGHOPPW}, the delta set of the shifted numerical monoid $\langle n, n+r_1, n+r_2,..., n+r_k\rangle$ is $\{1\}$. 

\item 
Let $M$ be a Krull monoid with finite divisor class group $G$ such that each divisor class of $G$ contains a prime divisor.  By \cite[Corollary 2.3.5]{G1}, the only two cases where $\Delta(M)=\{1\}$ occur are when $G=\mathbb{Z}_3$ or $G=\mathbb{Z}_2\oplus \mathbb{Z}_2$.   Thus, for instance, the elements of an algebraic ring of integers with class number 3 will satisfy Proposition \ref{delta1}.  A simplier construction is also possible.  Let $M=\{(x_1,x_2,x_3)\,\mid\, \mbox{each }x_i\in\mathbb{N}_0\mbox{ and } x_1+2x_2=3x_3\}$.
By \cite[Theorem 1.3]{CKO}, $M$ is a Krull monoid under addition with class group $\mathbb{Z}_3$ which has $\Delta(M)=\{1\}$. 

\item 
If a and b are positive integers with $a \leq b$ and $a^2 \equiv
a \bmod{b}$, then the set 
\[M_{a,b} = \{1\} \cup \{x \in \mathbb{N}\,\mid\,  x \equiv a \bmod{b}\}\]
is a multiplicative monoid known as an \textit{arithmetical congruence
monoid} (or \textit{ACM}).  If $\gcd(a,b)=p^\alpha$ for $p$ a prime, then we call $M$ \textit{local}.  Let $\beta$ be minimal with $p^\beta$ in $M$ and note that $\beta \ge \alpha$.  Then, by \cite[Theorem 3.1]{BCS}, if either $\alpha=\beta=1$ or if $\alpha<\beta\le 2\alpha$, then $\Delta(M)=\{1\}$.  So, for instance, all the elements of both $M_{4,6}$ and $M_{96,160}$ satisfy Proposition \ref{delta1}.  Thus, by \cite[Theorem 1]{BCCM2}, both $M_{4,6}$ and $M_{96,160}$ are monoids with accepted length density but not accepted elasticity.  

If $\gcd(a,b)>1$ is composite, then we call $M$ \textit{global}.  Among the various conditions given in \cite{BCS} which force a global ACM to have 
$\Delta(M)=\{1\}$ is Corollary 4.5 which shows that this is the case when $a=b$.  So while it is easy to verify that $\Delta(M)=\emptyset$ when $M=p\mathbb{N}\cup \{1\}$ for $p$ a prime, $\Delta(M)=\{1\}$ for monoids like $6\mathbb{N}\cup\{1\}$ and $9\mathbb{N}\cup\{1\}$.  Moreover, in the case where $a=b$ is composite and not a power of a prime, then it is easy to verify that $\rho(M)=\infty$.  Thus $0<\ld (M)$ does not imply that $\rho(M)<\infty$.
\item The last example motivates another class of monoids.  An atomic monoid $M$ is called \textit{bifurcus} if every nonzero, nonunit of $M$ can be factored into a product of two irreducibles.  By 
\cite[Theorem 1.1 (3)]{AAHKMPR}, for any $x\in M^{LI}$, $\mathsf{L}(x)=\{2,3,\ldots, L(x)\}$, so by definition, $\Delta(M)=\{1\}$. 
Various examples of rings that satisfy this condition can be found in \cite{AAHKMPR,BPAAHKMR}, including: 
\begin{enumerate}
\item 
$n\mathbb{Z}$ for $n$ not a prime power (bifurcus ring, without identity); 

\item 
$(m\mathbb{Z}) \times  (n\mathbb{Z})$ for $m,n$ each greater than 1 (bifurcus ring, without identity); 

\item 
the subring of $n\times n$ matrices consisting of matrices where all entries are identical integers and $n$ is not a prime power (bifurcus ring, without identity); and 

\item 
rank one matrices with entries from $\mathbb{N}$. 

\end{enumerate}
\end{enumerate}
\end{example}

We now provide a more general proposition.

\begin{proposition}\label{basicbounds}
If $x\in M^{LI}$, then 
\begin{equation}\label{fund1}
\frac{1}{\max \Delta(x)}\le \ld(x)\le \frac{1}{\min \Delta(x)},
\end{equation}
with equality on either side implying $|\Delta(x)|=1$, which in turn implies equality on both sides.  Furthermore, 
\begin{equation}\label{fund2} 
\frac{1}{\sup \Delta(M)}\le \ld(M)\le \frac{1}{\min \Delta(M)},
\end{equation}
with equality on the right side implying $|\Delta(M)|=1$, which in turn implies equality on both~sides.
\end{proposition}

\begin{proof}
Suppose $\mathsf{L}(x)=\{l_0,l_1,\ldots, l_k\}$ with $l_0<l_1<\cdots <l_k$.  Then
\[
\begin{array}{lcl}
L(x)-\ell(x) 
& = &  l_k-l_0 \\
& = & (l_k-l_{k-1})+(l_{k-1}-l_{k-2})+\cdots + (l_1-l_0) \\ & = & \sum_{i=1}^k l_i - l_{i-1}.\end{array}
\]
Since each summand is an element of $\Delta(x)$,
\[ k\min \Delta(x) \le \sum_{i=1}^k l_i - l_{i-1} \le k \max \Delta(x).\]
Dividing throughout by $|\mathsf{L}(x)|-1=k$, we get 
\[\min \Delta(x) \le \frac{L(x) - l(x)}{|\mathsf{L}(x)|-1}\le \max \Delta(x).\]
Taking reciprocals gives the double inequality in the theorem statement.  If either inequality is actually equality, then all summands are equal, and hence $|\Delta(x)|=1$.  The last inequalities now easily follow.
\end{proof}

Immediately we obtain the following.

\begin{corollary}\label{deltad} Let $M$ be an atomic monoid.  If $\Delta(M)=\{d\}$ for some positive integer $d$, then $\ld (x)=\frac{1}{d}$ and consequently $\overline{\ld}(x)=\frac{1}{d}$ for all $x\in M^{LI}$.  It follows that $\ld(M)=\frac{1}{d}$ and that the length density of $M$ is accepted.  
\end{corollary}

\begin{example}\label{extype2}
We offer some concrete examples to illustrate the last two results.
\begin{enumerate}
\item  
Let $M$ be a numerical monoid generated by an arithmetic sequence of integers (i.e., $M = \langle a, a+d, \ldots ,a+kd\rangle$ where $\gcd(a,d) = 1$ and $k<a$).   By \cite[Theorem 3.9]{BCKR}, $\Delta(M)=\{d\}$, and by Corollary \ref{deltad}, $\ld (x)=\overline{\ld}(x)=\frac{1}{d}$ for all nonunits $x\in M$.  As~such, $\ld(M)=\frac{1}{d}$. 

\item 
Let $G$ be an abelian group and $\mathcal{F}(G)$ the free abelian monoid on $G$.
We write the elements of $\mathcal{F}(G)$ in the form
$
X=g_1\cdots g_l = \prod_{g\in G} g^{v_g(X)},
$
The submonoid
\[
\mathcal{B}(G)=\big\{\!\textstyle\prod_{g\in G} g^{v_g}\in \mathcal{F}(G)\,\mid\,
\sum_{g\in G}v_g g=0\big\}
\]
is known as the \textit{block
monoid on} $G$ and its elements are referred to as \textit{blocks} over $G$.  If $S$ is a subset of $G$, then the submonoid
\[
\mathcal{B}(G,S)=\big\{\!\textstyle\prod_{g\in G} g^{v_g}\in \mathcal{B}(G)\,\mid\,
v_g=0\mbox{ if }g\not\in S\big\}
\]
of $\mathcal{B}(G)$ is called \textit{the restriction of
}$\mathcal{B}(G)$ \textit{to} $S$.  If $S$ generates $G$, then $\mathcal{B}(G,S)$ is a Krull monoid with divisor class group $G$.  If $n\geq 2$, then set
$\mathbb{Z}_n=\{\overline{0}, \overline{1},\ldots ,\overline{n-1}\}$. Consider the block monoid $\mathcal{B}(\mathbb{Z}_n,\{\overline{1},\overline{n-1}\})$.  It is easy to argue that the irreducible elements of $\mathcal{B}(\mathbb{Z}_n,\{\overline{1},\overline{n-1}\})$ are
\[
\overline{1}^{\; n}; \overline{n-1}^{\; n}; \overline{1}\;\overline{n-1},
\]
and the only relation amongst the irreducibles is $\overline{n-1}^{\; n} \cdot \overline{1}^{\; n}=(\overline{1}\; \overline{n-1})^n$.  By \cite[Lemma 2.8]{CS1}, 
$\Delta(\mathcal{B}(\mathbb{Z}_n,\{\overline{1},\overline{n-1}\})=\{n-2\}$, and hence $\ld(\mathcal{B}(\mathbb{Z}_n,\{\overline{1},\overline{n-1}\})=\frac{1}{n-2}$.

\item 
The bounds in Proposition \ref{basicbounds} may be strict in general.  Consider the numerical monoid $M=\langle 6,9,20\rangle$.  Here $\Delta(M)=\{1,2,3,4\}$.  
Using techniques from \cite{BOPa,CO}, it can be shown that $\ld(M)=\ld(60)=\frac{4}{7}$, and that $\ld(x)\to 1$ as $x\to \infty$.  

\end{enumerate} 
\end{example}

We now work toward the other extreme and start with a definition.

\begin{definition}\label{kainrathdef}
We say an atomic monoid $M$ has the \textit{Kainrath Property} if for every nonempty finite subset $L \subset \{2,3,4,\ldots \}=\mathbb{N}-\{1\}$ there exists an element $x\in M^{LI}$ such that $\mathsf{L}(x) = L$.
\end{definition}

Clearly a monoid $M$ with the Kainrath property satisfies $\Delta(M) = \mathbb{N}$.  
We also immediately deduce the following.  

\begin{corollary}\label{kainrathcor}
If $M$ has the Kainrath property, then 
$
\{\ld (x) \mid x\in M^{LI}\} = (0,1].
$
Hence, 
$
\ld(M) =0,
$
and $M$ does not have accepted length density.
\end{corollary} 

\begin{example}\label{kainrathex}
We now examine several families of Kainrath monoids.  
\begin{enumerate}
\item If $M$ is a Krull monoid with infinite divisor class group and a prime divisor in every divisor class, then $M$ satisfies the Kainrath property by \cite{Ka} and thus $\ld(M)=0$.  
The assumption that every divisor class contain a prime divisor is crucial here.  For instance, there are Krull domains with divisor class group $\mathbb{Z}$ which are half-factorial domains (see~\cite{ACS1}).  Recall that if $D$ is a Krull domain with class group $G$, then $D[X]$ is another Krull domain with class group $G$, but for $D[X]$ each ideal class of $G$ contains a prime divisor (see \cite[Theorem 14.3]{Fos}).  Thus, if $D$ is a Krull domain with infinite divisor class group, then $D[X]$ has the Kainrath property and hence $\ld (D[X])=0$ (even if these two facts do not hold for $D$).
\item An additive submonoid $P$ of $\mathbb{Q}_{\geq 0}$ (the nonnegative rationals) is known as a \textit{Puiseux monoid}.  These monoids are in some sense a natural generalization of numerical monoids.  If $S$ is a numerical monoid, then $|\Delta(S)| <\infty$ by \cite[Corollary 2.3]{BCKR}.  This fails in general for Puiseux monoids; in fact, by \cite[Theorem 3.6]{FG}, there exists a Puiseux monoid $P$ which has the Kainrath property.  Hence $\ld(P)=0$.   
\item 
In a recent paper Frisch \cite{Fr} showed 
$
\mathrm{Int}(\mathbb{Z}) = \{f(X)\in \mathbb{Q}[X]\,|\, f(z)\in \mathbb{Z} \mbox{  for all  }z\in \mathbb{Z}\},
$
called the ring of integer-valued polynomials over $\mathbb{Z}$, is Kainrath.  So, $\ld(\mathrm{Int}(\mathbb{Z}))=0$.
\item By \cite[Proposition~4.9 and Theorem~4.10]{FT} for every $n\in\mathbb{N}$, the power monoid $\mathcal{P}_{fin,0}(\mathbb{N})$ has an element $x_n$ with $\Delta(x_n)=\{n\}$; hence $\ld(x_n)=\frac{1}{n}$.  Consequently, we obtain $\ld(\mathcal{P}_{fin,0}(\mathbb{N}))=0$ but it is not accepted.  It is open whether or not $\mathcal{P}_{fin,0}(\mathbb{N})$
has the Kainrath property (see \cite[page~292]{FT}).
\end{enumerate}
\end{example}

We close this section by constructing two extremal monoids:\ one with rational length density that is not accepted (Example~\ref{nonacceptedld}), and another with infinite delta set but positive length density (Example~\ref{infinitedelta}).  

\begin{lemma}\label{directsumlemma}
If $x,y\in M^{LI}$ with $L(xy) = L(x) + L(y)$ and $l(xy) = l(x) + l(y)$, then
\[
\ld(xy)\ge \min(\ld(x),\ld(y)).
\]
Moreover, this inequality is strict if $\ld(x) \neq \ld(y)$.
\end{lemma}

\begin{proof}
We have $\mathsf{L}(xy)\supseteq \mathsf{L}(x)+\mathsf{L}(y)$, since we can always factor $xy$ by factoring $x$ and $y$ separately, and concatenating.  So, $|\mathsf{L}(xy)|\ge |\mathsf{L}(x)|+|\mathsf{L}(y)|-1$, by a simple observation on sizes of set sums ($|A+B|\ge |A|+|B|-1$).  Hence
\[\ld(xy)=\frac{|\mathsf{L}(xy)| -1}{L(xy)-l(xy)}\ge\frac{(|\mathsf{L}(x)|-1) + (|\mathsf{L}(y)|-1)}{(L(x)-l(x)) + (L(y)-l(y))}=\ld(x)\oplus \ld(y)\]
where $\oplus$ denotes the mediant $\frac{a}{b}\oplus\frac{c}{d}=\frac{a+c}{b+d}$.  A well-known property of the mediant is that $\frac{a}{b}\oplus\frac{c}{d}$ lies in the interval between $\frac{a}{b}$ and $\frac{c}{d}$ (in the interior unless $\frac{a}{b}=\frac{c}{d}$).
\end{proof}

\begin{theorem}\label{directsumthm}
For any collection of monoids $M_i$, we have
\[
\textstyle\ld\big(\!\bigoplus_i M_i \big) = \inf_i(\ld(M_i)).
\]
\end{theorem}

\begin{proof}
For any formal product $\prod_i x_i \in \bigoplus_i M_i$ of elements with each $x_i \in M_i$ and at least one $x_i \in M_i^{LI}$, the hypotheses of Lemma~\ref{directsumlemma} are satisfied, implying 
$\ld\big(\!\bigoplus_i M_i \big) \ge \inf_i(\ld(M_i))$.
Conversely, the image of any $x \in M_i$ in $\bigoplus_i M_i$ has identical length set, and thus equal length density.  This completes the proof.  
\end{proof}

\begin{example}\label{nonacceptedld}
Fix $i \ge 3$, and let~$M_i$ be the free abelian monoid on atoms $a_1, \ldots, a_i$, where
\[ a_1^3 = a_2^4 = a_3^6 = a_4^8 = \cdots = a_i^{2i} \]
are the minimal relations on the atoms.  Letting $M = \bigoplus_{i \ge 3} M_i$, it is clear that $\Delta(x) = \{1,2\}$ for every $x \in M$ with nonunique factorization, meaning the length density $\ld(M) = \tfrac{1}{2}$ obtained from Theorem~\ref{directsumthm} is not accepted.  
\end{example}

\begin{example}\label{infinitedelta}
For each $i \ge 2$, it is not hard to show, again using techniques from \cite{BOPa,CO}, that the numerical monoid $M_i = \<2i, 3i, 6i + 1\>$ has $\ld(M_i) = \tfrac{1}{2}$ achieved at $x_i = i(6i + 1) \in M_i$, which has length set 
$$\mathsf{L}_{M_i}(x_i) = \{i\} \cup \{2i+1, 2i+2, \ldots, 3i\}.$$
As such, $M = \bigoplus_{i \ge 2} M_i$ satisfies $|\Delta(M)| = \infty$ but $\ld(M) = \tfrac{1}{2} > 0$.  
\end{example}

\section{Nonrational and Accepted Length Density}
\label{sec:accepted}

A fundamental question early in the study of elasticity was whether or not an integral domain can have irrational elasticity.   In \cite[Theorem 3.2]{AndAnd} the authors show that for any real number $\alpha >1$, there is a Dedekind domain $D$ with $\rho(D)=\alpha$ (we note that if $\alpha\not\in \mathbb{Q}$, then $D$ must necessarily have infinite class group).   We now prove a somewhat similar result for length density, but use a completely different construction in the spirit of Examples \ref{nonacceptedld} and \ref{infinitedelta}. 

Let $a,b\in\mathbb{N}$ with $b>a$.  Let $c\in [0,1]$.  For each $i\in\mathbb{N}$, set $k(i)=\lceil ic(b-a)\rceil$.  We will now define the monoid $M(a,b,c)$, as  the free abelian monoid on atoms $\{q_{i,j}:i,j\in\mathbb{N}\}$, with minimal relations:
\[\forall i\in\mathbb{N}, ~~~~q_{i,ia}^{ia} = q_{i,ia+1}^{ia+1}=q_{i,ia+2}^{ia+2}=\cdots=q_{i,ia+k(i)}^{ia+k(i)}=q_{i,ib}^{ib}.\]

\begin{proposition}\label{niceconstruct}
If $a < b \in \mathbb{N}$ and $c\in [0,1]$, then $\rho(M(a,b,c))=\frac{b}{a}$ and $\ld(M(a,b,c))=c$.
\end{proposition}

\begin{proof} Set $M=M(a,b,c)$.  We first observe that for any $t\in\mathbb{N}$, we have 
\[
tic(b-a)+1\le tk(i)+1<tic(b-a)+1+t,
\]
and hence 
\begin{equation}\label{last}
c+\frac{1}{ti(b-a)}\le \frac{tk(i)+1}{ti(b-a)} < c+\frac{t+1}{ti(b-a)}.
\end{equation}
Note that each atom appears in at most one minimal relation.  We may thus calculate that $\mathsf{L}(q_{i,ia}^{ia})=\{ia,ia+1,ia+2,\ldots, ia+k(i),ib\}$.  We have $\rho(q_{i,ia}^{ia})=\frac{ib}{ia}=\frac{b}{a}$.  This proves that $\rho(M)\ge \frac{b}{a}$.  We also have $\ld(q_{i,ia}^{ia})=(k(i)+1)/(ib-ia)$.  By \eqref{last} with $t=1$, $\ld(q_{i,ia}^{ia})>c$ and also $\ld(q_{i,ia}^{ia})\to c$ as $i\to\infty$.  This proves that $\ld(M)\le c$.

Now, take $t\in\mathbb{N}$, and consider $\mathsf{L}(q_{i,ia}^{tia})$.  The minimum element is $tia$ and the maximum is $tib$. It contains the interval $[tia,tia+tk(i)]$, being the set sum of $t$ intervals.  Note that $tia+tk(i)<tib$.  Hence, $|\mathsf{L}(q_{i,ia}^{tia})|\ge tk(i)+2$, since the interval provides $tk(i)$ elements, and the right endpoint provides one more.  Hence, $\ld(q_{i,ia}^{tia})\ge\frac{tk(i)+1}{ti(b-a)}$.  By \eqref{last}, $\ld(q_{i,ia}^{tia})>c$.

Consider now an arbitrary element $x\in M$.  We write $x=x'\prod_{i\in\mathbb{N}} q_{i,ia}^{x(i)ia}$, where we choose $x(i)\in\mathbb{N}_0$ to be maximal. The leftover atoms, those dividing $x'$, are all inert, as $x'$~is the $\gcd$ of all factorizations of $x$.  We have 
\[\mathsf{L}(x)=\{|x'|\}+\sum_{i\in\mathbb{N}}\mathsf{L}(q_{i,ia}^{x(i)ia})\supseteq \{|x'|\}+\sum_{i\in\mathbb{N}}[x(i)ia, \, x(i)(ia+k(i))].\]
Now, the set sum of intervals is also an interval, so $\mathsf{L}(x)$ contains the interval \[
I=\bigg[|x'|+\sum_{i\in \mathbb{N}} x(i)ia, \, |x'|+\sum_{i\in\mathbb{N}}x(i)(ia+k(i))\bigg].
\]
Also,  $\mathsf{L}(x)$ contains, as maximum element, $|x'|+\sum_{i\in\mathbb{N}}ibx(i)$,  so we calculate
\[
\rho(x)=\frac{|x'|+\sum_{i\in\mathbb{N}}ibx(i)}{|x'|+\sum_{i\in\mathbb{N}}iax(i)}\le\frac{\sum_{i\in\mathbb{N}}ibx(i)}{\sum_{i\in\mathbb{N}}iax(i)} =\frac{b}{a}
\]
and, since $\mathsf{L}(x)$ contains the interval $I$, of length $|I|=1+\sum_{i\in\mathbb{N}} x(i)k(i)$,
\[
\ld(x)
\ge \frac{\sum_{i\in\mathbb{N}} x(i)k(i)}{\big(|x'|+\sum_{i\in\mathbb{N}}ibx(i)\big) - \big(|x'|+\sum_{i\in\mathbb{N}}iax(i)\big)}
\ge \frac{\sum_{i\in\mathbb{N}} x(i)ic(b-a)}{(b-a)\sum_{i\in\mathbb{N}}ix(i)}= c.
\]

Since $x\in M$ was arbitrary, this proves that $\rho(M)\le \frac{b}{a}$ and $\ld(M)\ge c$.  This establishes the desired result when combined with the opposite inequalities, established previously.\end{proof}


We begin to explore the question of when a monoid $M$ has accepted length density.  We first consider monoids with accepted length density equal to the left hand side of inequality~\eqref{fund2}.  Before proceeding, we will need some definitions.  
Suppose that $M$ is a commutative cancellative atomic BF-monoid.  Without loss for purposes of considering length sets, we assume that $M$ is reduced (i.e., has a unique unit).  Let $\mathcal{Z}(M)$ be the free abelian monoid on the atoms of $M$, which is called the \textit{factorization monoid of} $M$.  There is a natural map $\pi: \mathcal{Z}(M) \rightarrow M$ that sends a factorization on to its relevant element in $M$.    If $m \in M$, then set 
$
\mathsf Z_M(m)=\pi^{-1}(m)
$
which is known as the set of factorizations of $m$ in $M$.  Two factorizations $z,y\in \mathsf{Z}_M(m)$ can be written in terms of atoms as 
$z=u_1\cdots u_kv_1\cdots v_\ell$ and $y=u_1\cdots u_kw_1\cdots w_n$,
where $\{v_1,\ldots ,v_\ell\}\cap \{w_1,\ldots ,w_n\}=\emptyset$.  
Set $\gcd (z,y) = u_1\cdots u_k\in \mathcal{Z}(M)$,
and the distance between $z$ and $y$ to be $\dis (z,y)=\max\{n,\ell\}\in \mathbb{N}_0$.  
For each nonunit $m\in M$, the \emph{factorization graph}~$\nabla_m$ has vertex set $\mathsf{Z}_M(m)$, and two vertices $z,y \in \mathsf{Z}_M(m)$ share an edge if $\gcd(z,y)\neq 0$. If~$\nabla_m$ is not connected, then $m$ is called a Betti element of $M$. Write 
\[
\mathrm{Betti}(M) = \{b \in M^{LI} \mid \mbox{ if  }\nabla_b \mbox{  is disconnected}\}
\]
for the set of Betti elements of $M$.

\begin{proposition}\label{t:fingentastytest}
If $M$ has accepted length density, then $\ld(M) = 1/\max\Delta(M)$ if and only if $\ld(b) = 1/\max\Delta(M)$ for some $b \in \Betti(M)$.  
\end{proposition}

\begin{proof}
By the hypothesis, $|\Delta(M)| <\infty$, so let $\delta = \max\Delta(M)$.  The backward implication follows immediately from the bound $\ld(M) \ge 1/\delta$.  Conversely, suppose $\ld(M) = 1/\delta$.  By~assumption, there exists some $m \in M$ such that $\ld(m) = 1/\delta$.  Since $|\mathsf L(m)| \ge 2$, there exists $b \in \Betti(M)$ dividing $m$ in $M$ with $|\mathsf L(b)| \ge 2$, meaning $\mathsf L(b) + c \subseteq \mathsf L(m)$ for some $c \in \ZZ_{\ge 0}$.  In particular, $\Delta(m)$~must have some element at most $\max\Delta(b)$.  Since $\ld(m) = 1/\delta$, we see $\Delta(m) = \{\delta\}$, so we conclude $\Delta(b) = \{\delta\}$ as well.  
\end{proof}

\begin{example}\label{e:notatbetti}
Length density need not be attained at a Betti element.  Indeed, the numerical semigroup $M = \<20, 28, 42, 73\>$ has $\Betti(M) = \{84, 140, 146\}$ with length sets $\{2, 3\}$, $\{4, 5, 7\}$, and $\{2, 4, 5\}$, respectively, but $\mathsf L(202) = \{4, 6, 7, 9\}$ yields the length density $\tfrac{3}{5} < \tfrac{2}{3}$.  
\end{example}

\begin{theorem}\label{t:fingenaccepted}
If $M$ is finitely generated, then $\ld(M)$ is accepted.  
\end{theorem}

\begin{proof}
Suppose $m_1, m_2, \ldots \in M$ are distinct and that $\ld(m_i) > \ld(m_{i+1})$ for every $i$.  Since $M$ is finitely generated, there are only finitely many elements with any fixed factorization length, so by choosing an appropriate subsequence, it suffices to assume $|\mathsf L(m_i)| < |\mathsf L(m_{i+1})|$ for each~$i$.  

Now, in order for $\ld(m_i)$ to be a strictly decreasing sequence, there must exist a gap size~$\delta$ between sequential elements of $\mathsf L(m_i)$ satisfying $\ld(m_i) > 1/\delta$ that occurs arbitrarily many times in $\mathsf L(m_i)$ for $i$ large.  Since $M$ is finitely generated, the set of all trades with length difference at most $\delta$ is finite by \cite[Theorem~4.9]{factorhilbert}, so by omitting elements from the sequence $m_i$ as needed, we can assume there is some trade $a \sim b$ with $|b| - |a| = \delta$ between factorizations of some element $b \in M$ with disjoint support that can be performed at least $i$ times in $\mathsf Z(m_i)$ between factorizations of sequential length in $\mathsf L(m_i)$.  This in particular implies $b^i \mid m_i$ for all~$i$.  

If $\min \Delta(b) \ge \delta$, then we are done since $\ld(b) \le 1/\delta < \ld(m_i)$ for all $i$.  Otherwise, $\min \Delta(b) < \delta$, so the structure theorem for sets of length \cite[Theorem 8.1]{G2} implies there is a bound (independent of $i$) on the number of times $\delta$ occurs as a gap size between sequential elements of $\mathsf L(b^i)$, which is a contradiction since $b^i \mid m_i$.  
\end{proof}

\begin{example}\label{e:fingendescreasingsequence}
Infinite decreasing sequences of length density are possible in finitely generated semigroups, even though their length densities
are accepted.  To see this, consider the numerical semigroup $M = \<4,7\>$, so 
$$\mathsf L_S(28n) = \{ 4n, 4n+3, \ldots, 7n \}.$$
Then, take $T = \<(4,0,0), (7,0,0), (0,3,0), (0,1,1), (0,0,3)\>$, so 
$$\mathsf L_T( (28n, 3, 3) ) = \mathsf L_S(28n) + \{2,3\}$$
has gap sequence $1, 2, 1, 2, \ldots, 1, 2, 1$, producing a length density sequence
$$\ld_T( (28n, 3, 3) ) = (3n + 1)/(4n + 1)$$
that is strictly decreasing.  
\end{example}

\begin{example} 
We consider some examples related to Theorem \ref{t:fingenaccepted}.
\begin{enumerate}
\item If $G$ is a finite abelian group and $S$ a nonempty subset of $G$, then the block monoid $\mathcal{B}(G,S)$ is finitely generated and hence has accepted length density. 
For a monoid $M$, set $\mathcal{L}(M)=\{\mathsf{L}(x)\,\mid\, x\in M^{LI}\}$.  If $M$ is a Krull monoid with divisor class group $G$ and distribution of primes $S$, then by \cite[Proposition 12]{Ger1} we have that $\mathcal{L}(M)=\mathcal{L}(\mathcal{B}(G,S))$.  Hence, if $G$ is finite abelian, then $M$ has accepted length density.  This class includes all Krull domains with finite divisor class group, which includes the ring of algebraic integers in a finite extension of the rationals.
\item Since numerical monoids are finitely generated, they also have accepted length density.  While the computation of the elasticity of a numerical monoid $M$ is relatively simple \cite[Theorem 2.1]{CHM}, the computation of its length density is a more complex calcuation and we defer for the time being an extended study of this question.
\item An atomic integral domain $D$ is a \textit{Cohen-Kaplansky domain} (or a \textit{CK-domain}) if it has finitely many nonassociated irreducible elements.  In \cite[Theorem 4.3]{AM}, the authors give 14 conditions equivalent to $D$ being a CK-domain; the most notatable among these being $D$ is a one-dimensional semilocal domain such that for each nonprincipal maximal ideal $M$ of $D$, $D/M$ is finite and $D$ is analytically irreducible. As $\ld (px)=\ld (x)$ for any prime element $p\in D$ and nonzero nonunit 
$x\in D^{LI}$, Theorem \ref{t:fingenaccepted} implies that a CK-domain has nonzero accepted length density.
\end{enumerate}
\end{example}

We briefly approach the question of computing $\ld (\mathcal{B}(G))$ and start with a known result concerning the delta set of such a block monoid.  If $G=\sum_{i=1}^k \mathbb{Z}_{n_i}$ is a finite abelian group where $n_i |n_{i+1}$ for $1\leq i<k$ with $|G|\geq 3$, then by \cite[Corollary 2.3.5]{G1}
\begin{equation}\label{alfie}
[1,n_k-2]\subseteq \Delta(\mathcal{B}(G))\subseteq [1, \mathsf{c}(\block (G))-2]\subseteq[1, \mathsf{D}(G)-2].
\end{equation}
Here $\mathsf{D}(G)$ represents the \textit{Davenport Constant} of $G$; this is the longest length of a nonzero sequence which sums to $0$, but has no proper subsum that sums to zero.  The quantity $\mathsf{c}(M)$ is the \textit{catenary degree} of the monoid $M$, which we define generally as follows.   If $a\in M$ and $z_0, z_1,\ldots ,z_k\in \mathsf{Z}_M(a)$, then 
$z_0, z_1,\ldots ,z_k$ is called a \textit{chain from} $z_0$ to $z_k$.  For $N\in \mathbb{N}$,  $z_0, z_1,\ldots ,z_k$ is called an $N$-chain if $\dis (z_i,z_{i+1})\leq N$ for $0\leq i\leq k-1$.  The catenary degree, denoted $\mathsf{c}(a)$ of $a\in M$ is the smallest $N\in\mathbb{N}_0\cup\infty$ such that any two factorizations $z,y$ in $\mathsf{Z}_M(a)$ can be linked by an $N$-chain.  We then set $\mathsf{c}(M)=\sup\{\mathsf{c}(a)\,\mid\, a\in M\}$.  Equation \eqref{alfie} immediately leads to the following.

\begin{proposition}
  If $G$ is a finite abelian group with $|G|\geq 3$, then
 \[
  \frac{1}{\mathsf{c}(\block(G))-2}\leq \ld (\block (G)) \leq 1.
 \]
  \end{proposition}

It is known for each finite abelian group $G$ that $\Delta(\block (G))$ is a complete interval (i.e., $\Delta(\block (G))=\{1,2,\ldots , \max\Delta(\block (G))\}$) (see~\cite{GY1}). 
The containments 
\[
\Delta(\mathcal{B}(G))\subseteq [1, \mathsf{c}(\block (G))-2]\subseteq[1, \mathsf{D}(G)-2]
\]
form an equality if and only if $G$ is cyclic or an elementary 2-group \cite[Theorem~A]{GZ1}. 
Thus, in either of these two cases we have 
$
\Delta(\block (G))=\{1,2,\ldots, \mathsf{D}(G)-2\}$.
For a cyclic group, this yields that $\Delta(\block(\mathbb{Z}_n))=\{1,2,\ldots ,n-2\}$, and for an elementary 2-group it yields that 
$\Delta(\block (\sum_{i=1}^k \mathbb{Z}_2)) = \{1,2,\ldots , k-1\}$ (using the known formula for $\mathsf{D}(\sum_{i=1}^k \mathbb{Z}_2)=k+1$).  Combining this with Proposition \ref{basicbounds}, we obtain that
\[
\frac{1}{n-2}\leq \ld (\block (\mathbb{Z}_n)) \leq 1\mbox{  and  }
\frac{1}{k-1}\leq \ld (\block (\sum_{i=1}^k \mathbb{Z}_2))\leq 1.\]
By \cite[Corollary 2.3.6]{G1}, if $G$ is either cylic or an elementary 2-group, then some block $B\in\mathcal{B}(G)$ has $\mathsf{L}(B)=\{2,\mathsf{D}(G)\}$.  Hence,
in both cases
$\ld (B)=\frac{1}{\mathsf{D}(G)-2}$ which yields the~following.

\begin{proposition}\label{thisone}
If $G=\mathbb{Z}_n$ is cyclic, then $\ld (\block(\mathbb{Z}_n)) = \frac{1}{n-2}$, and if $G=\sum_{i=1}^k \mathbb{Z}_2$, then $\ld (\block (\sum_{i=1}^k \mathbb{Z}_2))= \frac{1}{k-1}$.
\end{proposition}

We list an application of Proposition \ref{thisone} to algebraic rings of integers.

\begin{corollary}
Any ring $R$ of algebraic integers with prime class number $p$ has 
$\ld(R) = \frac{1}{p-2}$.
\end{corollary}

In closing this section, we note that progress on the computation of further values of $\ld (\block (G))$ depends on improved computations of $\mathsf{c}(\block (G))$.  Outside of the groups 
listed in Proposition \ref{thisone}, $\mathrm{c}(\block (G))$ is known for only a handful of groups (see \cite[Theorems~A and~1.1]{GZ1}).  Hence, we leave an extended discussion of this problem to future consideration.


\section{Asymptotic Length Density}
\label{sec:asymptotic}

We open this section by constructing an atomic monoid with an element that lacks asymptotic length density.

\begin{example}\label{noasym}
Consider the Puiseux monoid  \[M=\left\langle \frac{4}{3}, \frac{8}{5},\frac{800}{1201},\frac{a_1}{p_1},\frac{a_2}{p_2},\ldots\right\rangle.\]  The $p_i$ are a strictly increasing sequence of primes, to be specified later.  The $a_i$ are a strictly increasing sequence of natural numbers, to be specified later.  Because the $a_i$ are distinct, all of the atoms of $M$ are unstable, so by \cite[Theorem 4.8]{FO}, $M$ is a BF-monoid (and an FF-monoid).

Our focus is on $x=8$, and we calculate $x^n=8n$, as $n$ grows large.
For $n<100$, $x^n<800$.  Hence, all factorizations of $x^n$ will include only the first two atoms. Note that $8=6\cdot \frac{4}{3}=5\cdot \frac{8}{5}$, so $\mathsf{L}(x^n)=\{5,6\}^n=[5n,6n]$.  In particular $\ld(x^n)=1$.

At $n=100$, $x^n=800$, and we get the new factorization $800=600\cdot \frac{4}{3}=500\cdot \frac{8}{5}=1201 \cdot \frac{800}{1201}$.  Note that $1201$ was chosen to be the smallest prime greater than twice $600$.  Hence we have $\mathsf{L}(x^{100})=[500,600]\cup \{1201\}$ so  $\ld(x^{100})<\frac{1}{2}$. 

As $n$ continues to increase, so long as $8n<a_1$, all factorizations of $x^n$ will include only the first three atoms.  Note that we have the trade $600\cdot \frac{4}{3}=500\cdot \frac{8}{5}=1201 \cdot \frac{800}{1201}$.  We take $n=100k$ and  calculate the length set of $x^{100k}$ as the union of intervals. We have 
\[
\mathsf{L}(x^{100k})=[500k,600k] \cup [500k+701,600k+601] \cup [500k+2\cdot 701,600k+2\cdot 601] \cup \cdots\, .
\]
The last interval will be $[500k+k \cdot 701,600k+k \cdot 601]=\{1201k\}$.
Note that if $600k\ge (500k+701)$, i.e. $100k\ge 701$, then the first two intervals overlap.  If $600k+601>500k+2\cdot 701$, i.e $100k\ge 801$, then the first three intervals overlap.  If $600k+i\cdot 601 > 500k+(i+1)\cdot 701$, i.e. $100k>701+100i$, then the $i$-th interval overlaps with the $i+1$-th interval.  Taking $i=\frac{3}{4}k$, if $k\ge 29>\frac{701}{25}$, we have the $i$-th interval overlapping with the $i+1$-th interval.  In particular, $\mathsf{L}(x)$ will contain the interval $[500k,600k+\frac{3}{4}k\cdot 601]$, i.e. $[500k, 1050.75k]$.  Hence $\ld(x^{2900})\ge \frac{550.75\cdot 29-1}{701\cdot 29}>\frac{3}{4}$.

We are now ready to choose the next atom.  Set $a_1=2901\cdot 8$, and $p_1>2\cdot 1201\cdot 30$, e.g. $p_1=72073$.  Using only the first three atoms, all factorizations of $x^{2901}$ are of length at most $30\cdot 1201$.  Using the new fourth atom, we get a new factorization of length $p_1$, which gives a gap of length at least $30\cdot 1201$.  Hence $\ld(x^{2901})<\frac{1}{2}$.

Continuing in this way, we find $\ld(x^n)$ can be made to grow to be above $\frac{3}{4}$, then to shrink  below $\frac{1}{2}$, over and over as $n\to \infty$.  Hence the asymptotic length density of $x$ does not exist.
\end{example}

What atomic monoids admit asymptotic length densities for all their elements?  While we do not completely answer this, we offer a large class that does.
We again require some definitions.  If $M$ is a monoid and $x\in M$, then let $\llbracket x\rrbracket$ denote the set of all elements in $M$ that divide $x^k$ for some $k\in\mathbb{N}$.
For $a\in M$ and $x \in \mathcal{Z}(M)$, let  $\mathsf{t}(a, x)\in \mathbb{N}_0-\{1\}$ denote the
smallest $N\in \mathbb{N}_0-\{1\}$ with the following property:
If $\mathsf{Z}_M(a) \cap x\mathcal{Z}(M) \neq \emptyset$  and $z \in \mathsf{Z}_M(a)$, then there exists some factorization
$z^\prime \in \mathsf{Z}_M(a) \cap  x\mathcal{Z}(M)$ such that $\dis (z, z^\prime) \leq N$.  If $\mathsf{Z}_M(a) \cap x\mathcal{Z}(M) = \emptyset$, then $\mathsf{t}(a,x)=0$.
We call $\mathsf{t}(a,x)$ the \textit{tame degree} of $a$ with respect to $x$.  If $u\in M$, then set $\mathsf{t}(M,u) =\sup\{\mathsf{t}(x,u) \mid x\in M\}$.  The monoid
$M$ is \textit{locally tame} if $\mathsf{t}(M,u)<\infty$ for each atom $u$ of $M$.  By \cite[Theorem 1.6.7]{GHKb}, if $M$ is an atomic locally tame monoid, then $M$ is a BF-monoid.
The \textit{tame degree} of $M$, is defined by $\mathsf{t}(M)=\sup\{\mathsf{t}(M,u) \mid u\in \mathcal{A}(M)\}$.  If $\mathsf{t}(M)<\infty$, then $M$ is called \textit{globally tame}.  
Since $\mathsf{c} (H) \le \mathsf{t} (H)$ by \cite[Theorem 1.6.6]{GHKb}, global tameness implies finiteness of the catenary degree.

\begin{theorem}\label{asymptotic}
Let $S$ be a locally tame atomic monoid and for $x\in S$ set $H=\llbracket x\rrbracket$.  Assume for $x\in S$ that
$\Delta(x)\neq \emptyset$ and $|\Delta(H)| <\infty$.  Let $d=\min \Delta(H)$, $\tau$ be minimal such that $d\in \Delta(x^\tau)$, $\psi=\max(\tau,\rho(\Delta(H))-1)$, and $T=\mathsf{t}(H,\mathsf{Z}_S(x^\psi))$.  
For all $n\ge \psi$ it follows that $\frac{1}{d}-\frac{2T}{nd^2} \le \ld(x^n) \le \frac{1}{d}$.  In particular, $\overline{\ld}(x)=1/d$.
\end{theorem}

\begin{proof}
Note that \cite[Theorem 1.6.7.2]{GHKb} implies the finiteness of $T$; $|\Delta(H)| <\infty$ yields that $\psi$ is finite. 
We first prove  that $H$ is locally tame.  First note that every unit of $S$ is in $H$; hence $H^\times=S^\times$ and thus $H_{red} \subseteq S_{red}$.  Because $S$ is locally tame, $\mathsf{t}(S,u)$ is finite for all $u\in \mathcal{A}(S_{red})$.
Hence $\sup \{\mathsf{t}(a,u): a\in S, u \in \mathcal{A}(S_{red}) \}$ is finite.
But $H\subseteq S$ and $H_{red} \subseteq S_{red}$. Hence $\sup \{\mathsf{t}(a,u): a\in H, u \in \mathcal{A}(H_{red}) \}$ is finite, being a supremum over a subset.
Hence $\mathsf{t}(H,u)$ is finite for all $u\in \mathcal{A}(H_{red})$.
Hence $H$ is locally tame.

Let $y\in H$ with $d\in \Delta(y)$.  That is, there are factorizations $z', z''$ of $y$ with $|z''|=|z'|+d$.  Since $y\in H$, there is some $m\in \mathbb{N}$ with $y \mid x^m$.  Let $z'''$ be any factorization of $y^{-1}x^m$.  Hence, $\mathsf{L}(x^m)$ contains $|z'''z''|$ and $|z'''z'|$, which are $d$ apart.  It cannot contain a factorization length in between, else $d$ would not be the minimum of $\Delta(H)$.  Hence $\Delta(x^m)$ contains $d$.  This proves that $\tau$ exists.

Because (see \cite[Definition 4.3.1]{GHKb}) $d\in \Delta(x^\tau)$, also $x^\tau\in \Phi(\{0,d\})$ (as is every higher power of $x$).  Recall that $\psi\ge \rho(\Delta(H))-1$.  We apply \cite[Theorem 4.3.6]{GHKb}, to get that every multiple of $x^\psi$ (in $H$) has its length set an AAP with difference $d$ and bound $\mathsf{t}(H,\mathsf{Z}_S(x^\psi))=T$.

Now take any $n\ge \psi$.  Apart from two intervals each of size at most $T$, every gap in $\mathsf{L}(x^n)$ is of size exactly $d$.  For every $n\ge \psi$, set $Q_n=\max \mathsf{L}(x^n) - \min \mathsf{L}(x^n)$.  We have 
\[
\frac{Q_n-2T}{d}\le |\mathsf{L}(x^n)|-1\le \frac{Q_n}{d},
\]
and so 
\[
\frac{1}{d}-\frac{2T}{Q_nd} \le \ld(x^n) \le \frac{1}{d}.
\]
Note that $Q_n\ge nd$, so 
\[
\frac{1}{d}-\frac{2T}{nd^2}\le \frac{1}{d}-\frac{2T}{Q_nd}.
\]
This completes the proof.
\end{proof}

Note the hypothesis that $|\Delta(H)| < \infty$ can be met by $\Delta(S)$ being finite.  This happens if the catenary degree is finite (see \cite[Theorem 1.6.3]{GHKb}) or, as noted above, if $M$ is globally tame.   

\begin{example}
We offer some examples that illustrate Theorem \ref{asymptotic}.
\begin{enumerate}
\item Finitely generated monoids are globally tame (see \cite[Theorem 3.1.4]{GHKb}), hence all nonunit elements admit asymptotic length densities.  
In particular, consider the numerical semigroup $M = \langle 6,9,20\rangle$ and $x=60$.  A brief calculation demonstrates that $\mathsf{L}(nx)=[3n,10n]\setminus\{3n+1,3n+2,3n+3\}$, so $\ld(x^n)=\frac{7n-3}{7n}=1-\frac{3}{7n}$.  Hence the rate of convergence (to the limit of 1), is exactly $O(1/n)$, proving the bound above is tight.
\item Let $H$ be a Krull monoid with class group $G$ and let $G_0 \subset G$ denote the set of classes containing prime divisors. If the Davenport constant $\mathsf D (G_0) < \infty$ (which holds if $G_0$ is finite), then $H$ is globally tame by \cite[Theorem 3.4.10]{GHKb}.  Thus such Krull monoids, such as the ring of algebraic integers in a finite extension of the rationals, satisfy Theorem~\ref{asymptotic}.
Suppose $H$ has infinite cyclic class group $G$, say $G = \mathbb{Z}$ and let  $G_0 \subset G$ denote the set of classes containing prime divisors. If $G_0 \cap \mathbb{N}$ or $G_0 \cap (-\mathbb{N})$ is finite, then $H$ is locally tame by \cite[Theorem 4.2]{Ge-Gr-Sc-Sc10} and has finite catenary degree by \cite[Theorem 1.1]{BCRSY} (we note that $H$ need not be globally tame).

\item Every C-monoid (see \cite[Definition 2.9.5]{GHKb}) is locally tame and has finite catenary degree by \cite[Theorem 3.3.4]{GHKb}.  Note that every Mori domain $D$ with nonzero conductor $\mathfrak f$, finite class group $\mathcal C (\widehat D)$ and finite residue field $\widehat D/\mathfrak f$ is a C-domain by \cite[Theorem 2.11.9]{GHKb}; orders in algebraic number fields are such Mori domains.

\end{enumerate}
\end{example}

\end{document}